\documentclass[12pt]{amsart}
\usepackage{amssymb,amsmath,epsf}
\usepackage{graphicx}
\usepackage{t1enc}
\usepackage[latin1]{inputenc}
\usepackage[german,english]{babel}
\usepackage{amsfonts}
\usepackage[all]{xy}
\usepackage{color}

\usepackage{geometry}
\newtheorem{defi}{Definition}[section]
\newtheorem{prop}[defi]{Proposition}
\newtheorem{theo}[defi]{Theorem}

\theoremstyle{remark}

\begin{document}
\title[Power commutator groups]{Power commutator groups}

\author[Arkadius Kalka]{Arkadius Kalka}

\address{Arkadius Kalka, 44328 Dortmund, Germany}
\email{arkadius.kalka@rub.de}

\begin{abstract}
We consider the class of finitely generated groups whose
relators are powers of commutators of the generators. This class contains as a
small subclass graph groups (also called RAAGs), namely if all powers are one.
Graph groups are the only torsionfree groups in this class.
The generators are of infinite order, but we may also add torsion by assigning
arbitrary orders to the generators. Then the above mentioned small subclass
contains partially commutative Shephard groups.
We show that these groups embed into Coxeter groups as finite index subgroups,
thus establishing the linearity of these groups.
The very short proof requires only elementary methods in combinatorial group
theory.
\end{abstract}



\maketitle


\section{Reidemeister-Schreier method}
Here we review briefly the Reidemeister-Schreier method.
Let $H$ be a subgroup of $G=\langle X \mid \mathcal{R} \rangle $, and let
$\mathcal{T}$ be  the \emph{Schreier transversal}  of right coset representatives of $H$ in $G$. 
 The \emph{Schreier map} $\gamma : \mathcal{T} \times X \longrightarrow H$ is defined by 
\[\gamma (t,x)=tx\cdot (\overline{tx})^{-1}. \] 
 Then we get the subgroup presentation $H=\langle \mathcal{Y} \mid \mathcal{R} \rangle $ with
 \emph{Schreier generator set} 
$\mathcal{Y}=\{ \gamma (t,x) \mid t \in \mathcal{T}, x\in X \}$ and
relator set $\mathcal{R}(H)=\{ \tau (trt^{-1}) \mid t\in \mathcal{T}, r\in \mathcal{R} \}$. 
Here  $\tau : F(X) \longrightarrow F(\mathcal{Y})$ denotes the \emph{Reidemeister rewrite map}
 which maps any freely reduced word $x_1x_2 \cdots x_l \in F(X)$ to a word over $\mathcal{Y}$ by
\[ \tau (x_1\cdots x_l)=\gamma(1,x_1)\cdot \gamma (\overline{x_1},x_2) \cdots \gamma (\overline{x_1\cdots x_{l-1}},x_l). \]

\section{Power commutator groups}
\begin{defi}
A \emph{power commutator (pc) group} is a f.g. group whose relators are powers of commutators of the generators, i.e of the form
$[g_i,g_j]^{n_{ij}}$ with $n_{ij} \in\{1,2,\ldots ,\infty \}$.
A \emph{finite order (f.o.) power commutator group} is a pc group with additional relations $g_i^{p_i}=1$ for all generators $g_i$ 
with $p_i \in \mathbb{N}\cup \{\infty \}$.
\end{defi}

{\bf Examples:} \,\, (1) \emph{Graph groups} (aka \emph{RAAGs} of \emph{partially commutative gps})
are the only torsionfree pc groups. In this case we have $n_{ij}\in \{1,2,\ldots \infty \}$ and $p_i=\infty $ for all $i\ne j$. \\
(2) If all powers (of commutators) are $1$ or $\infty $ 
then our f.o. pc group becomes a \emph{partially commutative Shephard group}. \\
(3) Quotients of fundamental groups of surfaces.

\begin{theo} \label{MainTh}
F.o. pc groups embed into Coxeter groups.
\end{theo}
\begin{proof}
 Let $W=W(D)$ be an \emph{even} Coxeter group with vertex set $S=(s_i)_{i\in I}$ ($|I|=n$), i.e.
its Coxeter diagram $D$ is given by a Coxeter matrix $(m_{ij})_{i,j\le n}$ with $m_{ij} \mod 2 \equiv 0$. 
Let $R$ be a copy of $S$.
Define $W''=W''(D'')$ with $Vert(D'')=R\cup S$, 
 and we add relations $(s_ir_i)^{p_i}$ $\forall i$ with $p_i \in \{2,3,\ldots , \infty \}$. 
Consider the homo $W'' \stackrel{\theta}{\longrightarrow \hspace{-16pt} \longrightarrow} \mathbb{Z}^n_2$ 
induced by $\theta : r_i, s_i \rightarrow r_i$.
Then $\ker \theta$ is a f.o. pc group, generated by $\{a_i\}_{i\in I}$, $a_i=s_ir_i$, 
with $[a_i,a_j]^{m_{ij}/2}=1$ and $a_i^{p_i}=1$ $\forall i,j$.
It remains to perform the explicit computation of $\ker \theta $.  \\
The Schreier transval for $\ker \theta$ is $\mathcal{T}=\{ \prod _{j \in J} r_j \mid J\subseteq I\}$ with 
$|\mathcal{T}|=2^n$. Evaluation of the Schreier map gives
$\gamma (t,r_i)=1$ $\forall (t,i)\in \mathcal{T}\times I$, and
 \[\gamma (t,s_i)=\left\{ \begin{array}{ll} s_ir_i=:a_i, & i\notin J, \\
                                         r_is_i=a_i^{-1}, & i\in J.
									       \end{array} \right., \]
i.e. we get as Schreier generator set $\mathcal{Y}=\{a_i\}_{i\in I}$. 
Consider the relations $r=(s_is_k)^{m_{ik}}$. First assume $i,k\notin J$ then 
\[\tau (trt^{-1})=(\gamma (t,s_i)\gamma (tr_i,s_k) \gamma (tr_ir_k,s_i)\gamma (tr_i,s_k))^{m_{ik}/2}=
(a_ia_ka_i^{-1}a_k^{-1})^{m_{ik}/2}. \] 
Similarly we get $\tau(r_itrt^{-1}r_i)=[a_i,a_k^{-1}]^{m_{ik}/2}$, $\tau (r_ktrtr_k)=[a_i^{-1},a_k]^{m_{ik}/2}$, 
and $\tau (r_ir_k(trt)r_ir_k)=[a_i,a_k]^{m_{ik}/2}$. Furthermore,
\[ \tau((s_ir_i)^{p_i})=(\gamma (1,s_i)\gamma (r_i,r_i))^{p_i}=(a_i\cdot 1)^{p_i}=a_i^{p_i}, \] 
and it remains to check that
$\tau(s_i^2)=\gamma (1,s_i)\gamma (r_i,s_i)=a_ia_i^{-1}=1$.
\end{proof}

{\bf Remark:}  
Graph groups (with $n_{ij}\in \{1, \infty \}$) embed into right angled Coxeter groups (with $m_{i,j}=2n_{ij}\in \{2,\infty \}$). This was shown by Davis and Januskiewicz \cite{DJ00}
using the action on cubical complexes. But the proof using the classical Reidemeister Schreier method is much shorter. \\

\section{Other embeddings into Coxeter groups}
\subsection{Coxeter groups}
We may also consider another natural homomorphism from $W''$ to an abelian group.
\begin{prop} \label{CoxEmb} 
Now, let $W=W(D)$ be any Coxeter group with vertex set $S=(s_i)_{i\in I}$
and index set $I$ of cardinality $|I|=n$. 
Define $W''$ as in the proof of Theorem \ref{MainTh} by attaching an abelian group $\langle R \rangle$ to $W(D)$,
but here we only allow \emph{even} $p_i$ for all $i$.
Consider the homo $\varphi : W'' \longrightarrow \hspace{-16pt} \longrightarrow \mathbb{Z}_2^n$  
defined by $r_i\mapsto r_i$  and
$s_i\mapsto 1$. 
Let $T=\{t_i=r_is_ir_i\}_{i\in I}$. Then
\[\ker \varphi =\langle S\cup T \mid s_i^2, t_i^2, (s_it_i)^{p_i/2}, (s_is_j)^{m_{ij}}, (s_it_j)^{m_{ij}},(t_it_j)^{m_{ij}} \, \, \forall i\ne j \rangle \] is also a Coxeter group which we denote by $W'$.
\end{prop} 

\begin{proof} The proof is a straightforward computation using the Reimeister-Schreier method:
The Schreier transversal for $\ker \theta$ is $\mathcal{T}=\{ \prod _{j \in J} r_j \mid J\subseteq I\}$ with 
$|\mathcal{T}|=2^n$. Evaluation of the Schreier map gives
$\gamma (t,r_i)=1$ $\forall (t,i)\in \mathcal{T}\times I$, and
 \[\gamma (t,s_i)=\left\{ \begin{array}{ll} s_i, & i\notin J, \\
                                         r_is_ir_i=t_i, & i\in J.
									       \end{array} \right., \]
i.e. we get Schreier generating set $\mathcal{Y}=\{s_i, t_i\}_{i\in I}$. 
Consider the relations $r=(s_is_k)^{m_{ik}}$. Then 
\[\tau (trt^{-1})=1\cdot (\gamma (t,s_i)\gamma (t,s_k))^{m_{ik}} \cdot 1=
\left\{ \begin{array}{ll} (s_is_k)^{m_{ik}}, & i,k \notin J,\\
                          (s_it_k)^{m_{ik}}, & i \notin J, k\in J,\\
													(t_is_k)^{m_{ik}}, & i\in J, k\notin J, \\
													(t_it_k)^{m_{ik}}, & i,k \in J.
         \end{array} \right. \] 
Furthermore,
\[ \tau((s_ir_i)^{p_i})=(\gamma (1,s_i)\gamma (1,r_i)\gamma (r_i,s_i)\gamma (r_i,r_i))^{p_i/2}=(s_it_i)^{p_i/2}, \] 
and 
$\tau(ts_i^2t^{-1})=(\gamma (t,s_i))^2=\left\{ \begin{array}{ll} s_i^2, & i \notin J, \\
                                                                 t_i^2, & i\in J. 
																			          \end{array} \right. $.
\end{proof}

Proposition \ref{CoxEmb} yields a lot of examples of embeddings of Coxeter groups, e.g., \\ 
$W(\tilde{A}_3)$ is an index 4 subgroup of $W(\tilde{C}_3)=W(\tilde{A}_3) \rtimes \mathbb{Z}_2^2$.

\subsection{Klein bottle group}

Recall the fundamental group $\pi _1({\rm Klein} \,\, {\rm bottle})=\langle a,b \mid a=bab \rangle $, and denote
by $D_{\infty }$ the infinite dihedral group.
\begin{prop}
$\pi _1({\rm Klein} \,\, {\rm  bottle})$ embeds into an right angled Coxeter group. More precisely,
We have the following split exact sequence: 
\[ \pi _1({\rm Klein} \,\, {\rm bottle}) \hookrightarrow D_{\infty }^2
\stackrel{4}{\longrightarrow \hspace{-16pt} \longrightarrow} \mathbb{Z}_2^2, \]
i.e $D_{\infty }^2= \mathbb{Z}_2^2 \ltimes \pi _1({\rm Klein} \,\, {\rm bottle})$. 
\end{prop}

\begin{proof}
Consider the direct product of two copies of infinite dihedral gps,
$D_{\infty }^2=X_{i=1,2}\langle r_i, s_i \mid r_i^2, s_i^2 \rangle $. 
Define the homomorphism $\theta' : D_{\infty}^2 \longrightarrow \mathbb{Z}_2^2$ by 
$s_1 \mapsto r_1r_2$, $s_2\mapsto r_2$, $r_i \mapsto r_i$ for $i=1,2$. Then the
Schreier generators of $\ker \theta' $ are $a_i=\gamma (1,s_1)$,  i.e. $a:=a_1=(s_1r_1)r_2$ and $b:=a_2=s_2r_2$. 
It turns out that the Reidemeister-Schreier process yields  (up to conjugation) only one nontrivial relator, namely 
\[ \tau((s_1s_2)^2)=\gamma (1,s_1)\gamma (r_1r_2,s_2)\gamma (r_1,s_1)\gamma (r_2,s_2)=ab^{-1}a^{-1}b^{-1}. \]  
Finally, we check for $i\ne j$ that $\tau ((s_ir_j)^2)=\gamma (1,s_i)\gamma (r_i,r_j)\gamma (r_ir_j,s_i)\gamma (r_j,r_j)=1$.
Hence, we have shown that $\ker \theta '=\langle a,b \mid a=bab \rangle $.
\end{proof}

\subsection{Artin groups}
We do not know whether there exists - besides graph groups - any other Artin groups that embed into Coxeter groups.
This seems to be open even for the braid group $\mathcal{B}_3$ or for the simplest $B$-type Artin group $\mathcal{A}(B_2)$.
But given some notion of generalized Coxeter groups, namely defined by labelled hypergraphs rather than labelled graphs,
one may embed Artin and Shephard groups into these generalized Coxeter groups. 

\begin{prop} 
Artin groups embed into a group generated by involutions. In particular, they embed into some generalized Coxeter groups. 
\end{prop} 
\begin{proof} Let $D$ be a Coxeter diagram corresponding to the Coxeter matrix $M=(m_{ij})_{i,j\in I}$.
As usually we have a copy $R$ of $S=\{s_i\}_{i\in I}$.
Defining words $w(i,j):=s_ir_is_jr_js_ir_i \cdots $ over $R\cup S$ with $|w|=2m_{ij}$ allows us to construct the following
generalized Coxeter group:
\[ W''=\langle R \cup S \mid r_i^2,s_i^2, [r_i,r_j], [r_i,s_j], w(i,j)w(j,i)^{-1} \,\, \forall i\ne j \rangle .\]  
Consider the homo $\theta :W'' \longrightarrow \hspace{-16pt} \longrightarrow \mathbb{Z}_2^n$ 
defined by $r_i,s_i \mapsto r_i$. 
Then $\ker \theta =\mathcal{A}(D)$ generated by $\{a_i=s_ir_i\}_{i\in I}$. 
It is straightforward to check that the Reidemeister-Schreier process yields (up to conjugacy) only the
following type of relations:
\[ \tau (w(i,j)w(j,i)^{-1})=(a_ia_ja_i\cdots)(a_ja_ia_j\cdots )^{-1}. \]
\end{proof}

{\bf Acknowledgements.}
This work was partially supported by the Emmy Noether Research Institute for
Mathematics and the Minerva Foundation (Germany).

\end{document}